\documentclass[11pt,a4paper]{extarticle}
\usepackage{microtype,amssymb,amsmath,amsthm,comment}
\usepackage[margin=1.3in]{geometry}
\usepackage[english]{babel}
\usepackage[utf8x]{inputenc}
\usepackage[colorinlistoftodos]{todonotes}
\usepackage[colorlinks=true, allcolors=blue]{hyperref}
\author{Johan Andersson\thanks{Email:johan.andersson@oru.se \, Address:Department of Mathematics, School of Science and Technology, {\"O}rebro University, {\"O}rebro, SE-701 82 Sweden. This work was partially funded by the Royal Swedish academy of sciences (A Gustaf Sigurd Magnuson grant).}}

\theoremstyle{plain}
\newtheorem{thm}{Theorem}
\newtheorem{cor}{Corollary}
\newtheorem{lem}{Lemma}

\theoremstyle{definition}

\newcommand{\nicefrac}[2]{\leavevmode\kern.1em
\raise.5ex\hbox{\the\scriptfont0 #1}\kern-.1em
/\kern-.15em\lower.25ex\hbox{\the\scriptfont0 #2}}

\def\halv{\mathchoice{{\textstyle{\frac 1 2}}}{1/2}{1/2}{1/2}}

\newcommand{\R}{{\mathbb R}} 

\newcommand{\abs}[1]{{\left| {#1} \right|}} \newcommand{\p}[1]{{\left(
      {#1} \right)}} 
\newcommand{\Oh}[1]{{O \p{#1}}}
\renewcommand{\Im}{\operatorname{Im}}
\newtheorem{vanlem}{Vanishing Lemma}
\newtheorem{prob}{Problem}

\renewcommand{\Re}{\operatorname{Re}}

\begin{document}

\date{}

\title{On the Balasubramanian-Ramachandra method close to $\Re(s)=1$}
\maketitle

\begin{abstract}
We study the problem on how to get good lower estimates for the integral
\begin{gather*}
 \int_T^{T+H} |\zeta(\sigma+it)| dt,
\end{gather*}
when $H \ll 1$ is small and $\sigma$ is close to $1$, as well as related integrals for other Dirichlet series, by using ideas related to the Balasubramanian-Ramachandra method. We use kernel-functions constructed by the Paley-Wiener theorem as well as the kernel function of Ramachandra.  We also notice that the Fourier transform of Ramachandra's Kernel-function is in fact a $K$-Bessel function. This simplifies some aspects of Balasubramanian-Ramachandra method since it allows use of the  theory of Bessel-functions.
 \end{abstract}

\tableofcontents

\section{Introduction and Main results}

 In a recent paper \cite{Andersson4} we proved that
\begin{gather*} 
  \inf_{|a_n| \leq \Phi(n)} \int_0^H \abs{1+\sum_{n=2}^N a_n n^{-it-1}} dt >0, \\ \intertext{if and only if}
 \int_2^\infty \frac{\log \Phi(x)} {x \log^2 x} dx < \infty,
\end{gather*}
whenever $\Phi(x)$ is an increasing positive function. This gives a strong answer to a question originally posed by Ramachandra \cite{Ramachandra2} (and solved in a weaker version in \cite{Andersson}), and has applications on lower bounds for the Riemann zeta-function close to $\Re(s)=1$. For example it implies that
\begin{gather} \label{aq}
   \int_T^{T+H} \abs{\zeta(\sigma+it)} dt \geq  C_{H,\varepsilon},   
   \qquad (1- \sigma <  (\log \log T)^{-\varepsilon-1}).
\end{gather}
To prove this we used integral kernels coming from a construction of Paley and Wiener \cite{PalWien}. We also used the following result:

\begin{vanlem} Any Dirichlet series that is identically zero on an interval of absolute convergence is identically zero on the complex plane. \end{vanlem}

In another direction Balasubramanian and Ramachandra devised a method (See for example \cite{Ramachandra}) which implies the following results:
\begin{align} \label{aq2}
    \max_{t \in [T,T+H]} \abs{\zeta(1+it)} &\gg  \log \log H,   \qquad (H \geq C_0 ), \\ \intertext{on the line $\Re(s)=1$,  for $1/2<\sigma<1$} \label{aq2b}    \max_{t \in [T,T+H]} \abs{\zeta(\sigma+it)} &\gg   \exp\p{\frac{C_\sigma(\log H)^{1-\sigma}}{\log \log H}}, \qquad (H \geq C_1 \log \log T), 
\end{align}
for some positive constants $C_\sigma,C_1$ as well as  other important results such as good omega-estimates (the same order of magnitude as the conjectured upper bounds) for higher power moments of the Riemann zeta-function on the critical line. It follows from an easy application of Voronin Universality, see e.g. \cite{Andersson,Andersson3,Andersson6} that
\begin{gather*}
  \inf_T \max_{t \in [T,T+H]} \abs{\zeta(\sigma+it)} =0,
\end{gather*}
for any $1/2<\sigma<1$ and thus in order for \eqref{aq2b} to be true, $H$ must be an increasing function of $T$. One of our aims is to prove new results that are in some sense intermediate to \eqref{aq2} and \eqref{aq2b}, when $\sigma$ is in the critical strip, but close to $1$. An important distinction is that our main interest lies in finding related results for small $H$, in particular when $H \to 0$. An example of such a result is the following theorem:

\begin{thm} 
  Let $T>16$ and $\varepsilon,C>0$. Then we have that
  \begin{gather*}
   \int_T^{T+H} \abs{\zeta(\sigma+it)} dt \gg \min(H^{2+\varepsilon},H), \, \, \, \, \, \,   \p{\sigma \geq 1-CH(\log \log T)^{-1-\varepsilon}}.
  \end{gather*}
 \end{thm}
In particular this result improves on \eqref{aq} by giving explicit estimates for $C_{H,\varepsilon}$ as $H$ tends to zero. A crucial part of the  proof of Theorem 1 is to use a standard Mollifier. It should be possible to use the same idea whenever we have an Euler-product. 

Although Theorem 1 extends to the line $\Re(s)=1$ it gives worse estimates in this case than our recent  \cite[Theorem 3]{Andersson3} surprisingly strong result 
 \begin{gather} \label{ity} \begin{split}
   \inf_T \int_T^{T+H} \abs{\zeta(1+it)} dt =\frac{e^{-\gamma}\pi^2}{24}H^2+\Oh{H^4}, \\ \inf_T \int_T^{T+H} \abs{\zeta(1+it)}^{-1} dt =\frac{e^{-\gamma}}{4}H^2+\Oh{H^4}, \end{split}
  \end{gather}
when $H \to 0^+$.  By continuity however, this theorem can be extended to some sufficiently small region in the critical strip. The Riemann hypotheses together with this result infact implies a stronger result than Theorem 1. Unconditionally we may replace the Riemann hypothesis with the sharpest known zero-free regions to obtain sharper bounds than in Theorem 1 at the expense of a shorter range of $\sigma$.

\begin{thm}
 Assuming the Riemann hypothesis\footnote{In fact the so called quasi Riemann hypothesis suffices. There is some constant $c<1$ such that the Riemann zeta-function has no zeroes for $\Re(s)>c$.} one has that
 \begin{gather*}
   \liminf_{\substack {T \to \infty \\ 1-\omega(T) \leq \sigma \leq 1}} \int_T^{T+H} \abs{\zeta(\sigma+it)} dt =\frac{e^{-\gamma}\pi^2}{24}H^2+\Oh{H^4}, \\ \liminf_{\substack {T \to \infty \\ 1-\omega(T) \leq \sigma \leq 1 }} \int_T^{T+H} \abs{\zeta(\sigma+it)}^{-1} dt =\frac{e^{-\gamma}}{4}H^2+\Oh{H^4}, \\ \intertext{for $0<\omega(T)<1$ such that} 
 \omega(T)= o \p{\frac 1 {\log \log T}}.  \\ \intertext{Unconditionally the same result holds when}
  \omega(T)=o((\log T)^{-2/3}(\log \log T)^{-1/3}). \end{gather*}
\end{thm}

\begin{proof}
 The conditional result is a consequence of \eqref{ity} and  Titchmarsh \cite[p. 383, last equation]{Titchmarsh}\begin{gather*}  
 \frac{\zeta'(\sigma+it)} {\zeta(\sigma+it)} \ll \frac{(\log t)^{2-2 \sigma}-1}{1-\sigma}, \qquad (1/2<\sigma_0\leq \sigma < 1, \text{ Assuming RH}), \\ \intertext{from which it follows that}
  \abs{\zeta(1+it)-\zeta(\sigma+it)} \ll \delta \abs{\zeta(1+it)},\qquad ( 1-\delta/\log \log t \leq \sigma \leq  1).
\end{gather*}
The unconditional result follows in a similar way by using the unconditional result
 \begin{gather*}
 \frac{\zeta'(\sigma+it)} {\zeta(\sigma+it)} \ll (\log t)^{2/3} (\log \log t)^{1/3}, \, \, \, \, \,  (1-A (\log t)^{-2/3} (\log \log t)^{-1/3} \leq \sigma  \leq 1),
\end{gather*}
 see e.g. the discussion by Heath-Brown \cite[p.135]{Titchmarsh}. This estimate is a consequence of the unconditional zero free regions of Vinogradov \cite{Vinogradov} and Korobov \cite{Korobov}.   The strongest constants in the zero free region is due to  Ford \cite{Ford}. 
From this inequality it follows that 
\begin{gather*}
  \abs{\zeta(1+it)-\zeta(\sigma+it)} \ll \delta \abs{\zeta(1+it)}, \,  \, \, \, ( 1-A\delta/((\log t)^{2/3}(\log \log t)^{1/3}) \leq \sigma \leq  1),
\end{gather*}
from which the unconditional result follows.
\end{proof}

For the case when we do not have an Euler product we obtain even weaker estimates in $H$.  However, by a quantitative variant of the vanishing Lemma proved in \cite{Andersson15} we will be able to treat this case as well. Also, instead of the construction of Paley and Wiener, we will use  the integral Kernel introduced by Ramachandra (see for example \cite{Ramachandra}), which gives the sharpest range in the Balasubramanian-Ramachandra method.  We choose as a prototype case the Hurwitz zeta-function, although the result can be proved in a more general context, like that of Titchmars series.  For the Hurwitz and Lerch zeta-functions we have the following result
\begin{thm} 
 Let $\zeta(\sigma+it,\alpha)$ with $0<\alpha \leq 1$ be the Hurwitz zeta-function. Then for $T \geq 16$ one has that
\begin{gather*}
 \int_{T}^{T+H} \abs{\zeta(\sigma+it,\alpha)}dt   \gg  \min\p{H,\p{\frac {H} {200}}^{\frac 7 {6H \varepsilon}}}, \\ \intertext{whenever}
     \sigma \geq  1-\frac{\pi H(1-\varepsilon)}{4 \log \log T},
\end{gather*}
and $0<\varepsilon \leq 1$. Furthermore the same estimate is valid when the Hurwitz zeta-function $\zeta(\sigma+it,\alpha)$ is replaced by the Lerch zeta-function $\phi(\alpha,\beta,\sigma+it)$ with $0<\alpha,\beta  \leq 1$.
\end{thm}
We remark that this result not just gives an explicit estimate for \eqref{aq} but the use of the integral kernel of Ramachandra allows us to  improve on the range for $\sigma$ where it is valid. Thus it also gives stronger estimates for the Riemann zeta-function case than Theorem 1 when a wider range of $\sigma$ is considered, at the expense of a obtaining a weaker lower estimate in $H$. Assuming the Riemann hypothesis however, the same arguments used to prove Theorem 2 gives stronger results.

\section{The multiplicative case}

We will infact prove a stronger result, from which Theorem 1 is an immediate consequence.

\begin{thm}  
  Suppose $x\omega(x)$ and $1/\omega(x)$ are  increasing positive functions  for $x \geq 1$ such that $\omega(x)=1$ for $0 \leq x \leq 1$ and
 \begin{gather*}
   \int_1^\infty \frac {\omega(x) dx } {x}<\infty. 
 \end{gather*}
   Then one has for $0 \leq H \leq 1$ and $T\geq 1$  that
  \begin{gather*}
   \int_T^{T+H} \abs{\zeta(\sigma+it)} dt \gg  \frac{H^2  \omega(\abs{\log H})}{1+\abs{\log H}}, \qquad  \p{\sigma \geq 1-  H \omega\p{H \log T}}.
  \end{gather*}
 \end{thm}
We first prove a Lemma:
\begin{lem}
 Let $\sigma, T,H$ be given as in Theorem 4. Then there exist a positive test function  $\phi \in C_0^\infty(\R)$ with support on $[0,1]$ such that $\phi(0)=c_0>0$, $0 \leq \phi(t) \leq 1$ and 
 \begin{gather*}
    \abs{\hat \phi(H \log n) n^{-\sigma}} \leq (\log n)^{-3} n^{-1}, \qquad (X \leq n \leq T^2), \\ \intertext{where}
     X=  \exp\p{\frac{\abs{\log H}} {H   \omega(\abs{\log H})}}
 \end{gather*}
\end{lem}
\begin{proof}
 By using the inequality $\sigma \geq 1-H\omega(H\log T)$ and the fact that $\omega(x)$ is decreasing it is clear that
\begin{gather}
  \label{sde}
  \frac {1-\sigma} H-\omega(x/2) \leq  \frac {H \omega (H \log T)} H -
   \omega(H \log T) = 0, \qquad (0 \leq x \leq 2H \log T). 
\end{gather}
By the Paley-Wiener theorem's \cite{PalWien}, see also Koosis \cite{Koosis} or for a suitable version see \cite[Lemma 4]{Andersson}, we can choose a positive test function  $\phi \in C_0^\infty(\R)$ with support on $[0,1]$ so that $\hat \phi(0)=c_0 \neq 0$  and such that
\begin{gather} \label{phidef}
 \abs{\hat \phi(x)} \leq    x^{-3} \Phi(x)^5, \qquad \text{where} \qquad \Phi(x)=\exp \p{-x \omega(x/2)}. 
\end{gather} 
From the requirement that $x \omega(x)$ is an increasing function in $x$ we have that $\Phi(x)$ is a decreasing function. It is clear that
$$H \log X = \frac{\abs{\log H}} {\omega(\abs{\log H})}.$$ 
Since $\omega(x)$ is a positive decreasing function for $x \geq 1$ such that $\omega(x)=1$ for $0 \leq x  \leq 1$  it follows that $\omega(x) \leq 1$ for $x \geq 0$ and that $H \log X \geq \log H$. It follows that $\omega(H \log X) \leq \omega(\abs{\log H})$.
Thus we see that
\begin{gather*} 
     \Phi(H \log X) =  \exp \p{- \frac{\abs{\log H}}{\omega(\abs{\log H)}} \, \omega(H \log X)} \leq \exp \p{-\abs{\log H}} = H.
\end{gather*}
Since $\Phi$ is a decreasing function we see that
$\Phi(H \log n)^3 \leq H^3$ for  $n \geq X$ and we obtain
\begin{gather} \label{aaa} (H \log n)^{-3}  \Phi(H \log n)^3 \leq (H \log n)^{-3} H^3 =  (\log n)^{-3}, \qquad (n \geq X). \end{gather}
By \eqref{phidef} it is clear that
\begin{gather} \notag
   n^{-\sigma} \Phi(H \log n)^2 \leq n^{-1}  \exp \p{\frac{(1-\sigma)x}{H} -2x \omega(x)}, \\ \intertext{where $x=H \log n$. Since $1 \leq n\leq T^2$ this means that $0 \leq x<2H \log T$ and we can use the inequality \eqref{sde} to obtain}
 n^{-\sigma} \Phi(H \log n)^2\leq n^{-1}, \label{bbb}
\end{gather}
The lemma follows from combining the inequalities \eqref{phidef}, \eqref{aaa} and \eqref{bbb}.
\end{proof}

\subsection{Proof of Theorem 4}

By a suitable approximate functional equation 
(Ivi{\'c} \cite[Theorem 1.8]{Ivic}) we have that
\begin{gather} \label{approx}
\zeta(\sigma+it+iT)=\zeta_T(\sigma+it)+O(T^{-\halv}),  \qquad (T/2<|t|<T,\sigma \geq 1/2), \\ \intertext{where} \zeta_T(s) =\sum_{1 \leq n < T} n^{-s}. \notag
\end{gather} 
Thus it will be sufficient to consider Dirichlet polynomials instead of Dirichlet series. Let $X$ be defined as in Lemma 1  and introduce the standard Mollifier\footnote{This has also been used by Selberg \cite{Selberg} to show a positive proportion of zeros on the critical line and has also had other important applications such as zero density estimates (See Ivi{\'c} \cite{Ivic}, chapter 11).}:
 $$M_X(s)=\sum_{1 \leq n \leq X} \mu(n) n^{-s}.$$ 
 Without loss of generality we may assume that $X<T$. Define
 \begin{gather} \label{Adef} A(s)=\zeta_T(s+iT) M_X(s+iT)=\sum_{1 \leq n <T^2} a_n n^{-s}. \end{gather}
   It is clear that
 \begin{gather} a_n=\begin{cases} 1, & n=1,\\ 0, & 2 \leq n < X, \end{cases} \\ \intertext{and that} \label{aineq} |a_n| \leq d(n).
\end{gather} 
Now let $\phi(x)$ be the function in Lemma 1. By the definition of the Fourier-transform 
$$
 \hat \phi(x)=\frac 1 {2 \pi}\int_{-\infty}^\infty e^{-itx} \phi(t) dt,
$$
it follows that
$$ \int_{-\infty}^\infty \phi\p{\frac t H} A(\sigma+it)  dt = 2 \pi H \sum_{1 \leq n <T^2} a_n n^{-\sigma} \hat \phi(H \log n).
$$
 Hence by \eqref{aineq} and Lemma 1 we obtain 
\begin{gather} \label{aboj}
  \sum_{n=1}^{T^2} a_n n^{-\sigma}  \hat\phi(H \log  n)= 2 \pi c_0 H + \Oh{H \sum_{X \leq n <T^2} d(n) (\log n)^{-3} n^{-1}}. 
\end{gather}
From the fact that
 \begin{gather} \notag
\zeta^2(s)=\sum_{n=1}^\infty d(n) n^{-s} \\ 
\intertext{is analytic for $\Re(s) \geq 1$, except for a second order pole at $s=1$,  it follows that} 
\label{ina2} \sum_{n>X} \frac{d(n)} n (\log n)^{-3} \ll (\log X)^{-1},  \end{gather}
and  from \eqref{aboj} and \eqref{ina2}  and the choice of $X$ given in Lemma 1 we see that  
\begin{gather*} 
  \int_{-\infty}^\infty \phi \p{\frac t H} A(\sigma+it)  dt = 2 \pi c_0 H + 
\Oh{\frac{H \omega(\abs{\log H})}{\abs{\log H}}}.
\end{gather*}
Since $\omega(x) \leq 1$ and $\lim_{H \to 0^+} \abs{\log H}=\infty$ it it is clear that for sufficiently small $0<H \leq H_0$ the error term will be less than half of the main term and by the triangle inequality and the fact that $\phi$ has support on  $[0,1]$  it follows that
\begin{gather} \label{ajaj}
 \int_0^H \abs{A(\sigma+it)} dt \geq \pi c_0 H , \qquad (0<H \leq H_0).
\\ \intertext{By the definition of $A(s)$, Eq.  \eqref{Adef} it is clear that}
\label{ajajaj} \int_T^{T+H} |\zeta_T(\sigma+it)| dt \geq \frac{\int_0^H |A(\sigma+it)| dt}{\max_{t \in [T,T+H]} |M_X(\sigma+it)|}. \\ \intertext{From the triangle inequality we have}
 \label{ajajajaj}
\abs{M_X(\sigma+it)} \leq  \sum_{n=1}^X n^{-\sigma} \ll \max\p{\frac{X^{1-\sigma}-1}{1-\sigma},\log X} \ll  \frac{\abs{\log H}} { \omega(\abs{\log H})H}.
\end{gather}
Our result for $0<H \leq H_0$ thus follows from the approximate functional equation \eqref{approx}, and  the inequalities \eqref{ajaj}, \eqref{ajajaj} and \eqref{ajajajaj}. The result for $0<H_0<H \leq 1$ is a trivial consequence of the result for $H=H_0$. \qed

\subsection{Proof of Theorem 1}
For the case $0<H<1$ Theorem 1 follows  by choosing
\begin{gather*}
 \omega(x)=\begin{cases} 1, & 0 \leq x \leq 1, \\ \p{1+\log x}^{-1-\varepsilon}, & x>1, \end{cases} \\ \intertext{in Theorem 4 and by the fact that}
 \abs{\log H}^{-1-\varepsilon} \ll H^{\varepsilon}, \qquad (H<1/2). 
\end{gather*}
The case $H \geq 1$ follows from Theorem 3. \qed
\section{An Integral kernel of Ramachandra}

\subsection{An optimal kernel}

We will use the same test-function as Ramachandra \cite[p. 35]{Ramachandra}, although we will treat it somewhat differently. Ramachandra used the test-function $\exp(\sin^2 w)$. He then proved some results on the Fourier transform of this function. Ivi{\'c} (\cite{Ivic} and \cite[pp. 21-22]{Ivic2})  considered the function $\exp(- \cos w)$ instead, which by the trigonometric identity $\cos(2x)=1-2\sin^2(x)$ is essentially equivalent. This test-function of Ramachandra is in fact optimal in a certain sense.  We quote from Ivi{\'c} \cite[p. 22]{Ivic}:

``In part I of \cite{Ramachandra6} Ramachandra expresses the opinion that probably no function regular in a strip exists, which decays faster than a second order exponential. This is indeed so, as was kindly pointed out to me by W. K. Hayman in a letter of August 1990. Thus Ramachandras kernel function $\exp(\sin^2 w)$ (or $\exp(-\cos w)$) is essentially the best possible.''

\subsection{$K$-Bessel functions}
 Instead of treating this function directly, we  relate this kernel to the theory of the Macdonald, or $K$-Bessel function $K_\nu(z)$ introduced by 
Basset \cite{Basset} for integer values of $\nu$ and generalized to noninteger values of $\nu$ by Macdonald \cite{Macdonald}. Sch\"afli \cite{Schafli} proved\footnote{Since Sch\"afli's results predates the introduction of the $K$-Bessel-function he used a different notation in his paper.} that
\begin{gather} \label{ttt}
  K_{\nu}(x)=\halv e^{-\halv \nu \pi i} \int_{-\infty}^\infty e^{-ix \sinh t -\nu t}dt.
\end{gather}
For this result see Watson \cite[6.22, Eq. (10)]{Watson}. The explicit relationship between this integral and Ramachandra's kernel function will be given by Theorem 5. By noticing the connection with the Bessel functions,   the required results for Ramachandra's  kernel-function needed to develop the Balasubramanian-Ramachandra method are simple consequences of  well-known results from the theory of Bessel function.

\subsection{A summation formula and $K$-Bessel functions}

\begin{thm}
 Let \begin{gather*} A(s)=\sum_{n=1}^\infty a_n n^{-s} \\ \intertext{be a Dirichlet series absolutely convergent for $\Re(s)>\sigma$. Then for $x,\lambda>0$ and $\Re(s)>\sigma$. we have that}
 \sum_{n=1}^\infty a_n K_{i \lambda \log n}(x) n^{-s} =  \frac {1} {2 \lambda} \int_{-\infty}^{\infty } A(s+it) e^{-x \cosh (t/\lambda)} dt.
\end{gather*}
\end{thm}
\begin{proof}

  By using $\nu=i \mu$ and moving the first factor inside the integral, Sch\"afli's identity Eq.  \eqref{ttt} can be written as
 \begin{gather} \notag
     K_{i \mu}(x) =  \frac 1 2   \int_{-\infty}^{\infty} e^{-ix \sinh t- \mu i (t+\pi i /2)}dt. \\ 
 \intertext{With the substitution $\tau=t+\pi i/2$ this integral equals} \notag
    \frac 1 2   \int_{\pi i/2-\infty}^{\pi i/2+\infty}  e^{-ix \sinh (\tau-\pi/2 i)- \mu i \tau} d\tau. \\ \intertext {By moving the integration line from $\Im(\tau)=\pi/2$  to $\Im(\tau)=0$ and using the identity $\cosh \tau= \sinh (\tau-\pi/2 i) i$ we find that}
     K_{i \mu}(x) = \frac{1}{2}  \int_{-\infty}^{\infty}  e^{-x \cosh \tau- \mu i  \tau} d\tau. \label{Kdef}
\end{gather}
Applying this term-wise and interchanging the summation and integration gives us the identity
\begin{align*}
 \sum_{n=1}^\infty a_n K_{i \lambda \log n}(x) n^{-s} 
&=\sum_{n=1}^\infty a_n n^{-s} \frac 1 2  \int_{-\infty}^{\infty}  e^{-x \cosh \tau- \lambda \log n \tau} d\tau, \\ 
&= \frac{1}{2}  \int_{-\infty}^{\infty}  e^{-x \cosh \tau} A(s+\lambda \tau i) d\tau. 
\end{align*}
By the substitution $t=\lambda \tau$ this equals
\begin{gather*}
\frac{1}{2 \lambda} \int_{-\infty}^{\infty}  e^{-x \cosh (t/\lambda)} A(s+it) dt. 
\end{gather*}
.\end{proof}

\subsection{Asymptotic estimates for $K$-Bessel functions}
It will be sufficient for us to use Theorem 5 for some fixed $x>0$. For convenience we will state the following lemma for $x=2$ although a similar result can be proved for arbitrary $x$ as well:

\begin{lem}
   We have for $t>0$ that
    \begin{gather*} 
       K_{it}(2)=  e^{-\pi t/2} \sqrt{ \frac {2 \pi} {t}}  \sin \p{\frac \pi 2 \p{ t \log t+t}}   \p{1+\Oh{\frac 1 t}}. 
    \end{gather*}
 \end{lem} 
\begin{proof}\footnote{This result is most likely well-known and in such a case I should find a reference. I could not find the result in Watson \cite{Watson}, and since its proof is simple it might as well remain even if I find a reference.}
Similarly to the asymptotic expansion of $J_\nu(z)$ made by Watson \cite[Section 8.1]{Watson},  it follows from the definition of the $K_\nu(z)$-Bessel function (Watson \cite[Section 3.17 $(6)$ and $(7)$]{Watson}.
\begin{gather*}
  K_\nu(z)=\frac {\pi} 2 \frac{I_{-\nu}(z)-I_\nu(z)}{\sin \nu \pi}, \\ \intertext{and}
  I_\nu(z)= \sum_{m=0}^\infty \frac{(z/2)^{\nu+2m}}{m! \Gamma(\nu+m+1)},
\end{gather*}
that
\begin{gather*} 
  K_\nu(x)= \frac {\pi} {2\sin \nu \pi} \p{\frac {(x/2)^{-\nu}} {\Gamma(1-\nu)}-\frac {(x/2)^{\nu}} {\Gamma(1+\nu)}} \p{1+O_x \p{\frac 1 \nu}}. \\ \intertext{By the reflection formula for the Gamma-function}
  \Gamma(z)\Gamma(1-z) =\frac{\pi}{\sin \pi z}, \\ \intertext{and the fact that $\Gamma(z+1)=z \Gamma(z)$, this  simplifes to}
  K_\nu(x)= \frac{1} {2\nu}  \p{(x/2)^{-\nu}\Gamma(1+\nu)- (x/2)^\nu \Gamma(1-\nu)} \p{1+O_x \p{\frac 1 \nu}}.
\end{gather*}
The final result follows from Stirling's formula
\begin{gather*}
  \Gamma(z)= \sqrt{2\pi/\nu} \p{z/e}^z \p{1+\Oh{1/z}},
\end{gather*}
and letting $\nu=it$.
\end{proof}

\section{The non multiplicative case}

\subsection{Some lemmas}
\begin{lem}
  Suppose $\lambda,T>0$. Then
 \begin{gather*}
    \int_{|t| \geq T} \exp \p{-2\cosh \frac t \lambda}  dt \leq  \frac 1 \lambda \exp \p{-\exp\p{\frac{T} {\lambda}}}.
   \end{gather*}
\end{lem} 

\begin{proof} 
  By the substitution $\tau=t/\lambda$ it is sufficient to prove the Lemma in case $\lambda=1$, i.e. 
\begin{gather}
    \int_{|\tau| \geq T} \exp \p{-2\cosh \tau}  d\tau \leq    \exp \p{-\exp\p{T}}. 
\end{gather}

This result  is trival for large $T$ and follow from numerical estimation in Maple\footnote{This should  possibly be done in a more rigid manner (without Maple).} for small $T$ (infact replacing the factor 1 in the RHS by $2e K_0(2)=0.619$. gives the optimal bound).
\end{proof}

\begin{lem}
  Suppose $f(t)$ is some function such that $ \abs{f(t)} \leq C$ for all real $t$ and that for some $\sigma<1$ and $\varepsilon>0$ we have the inequality
\begin{gather*}
   \int_{-\infty}^\infty \abs{f\p{t-\frac H 2}} \exp \p{-2 \cosh\p{\frac{\pi t} {2(1-\sigma)}}} dt \geq \varepsilon.
 \\ \intertext{Then} 
   \int_{0}^H \abs{f(t)} dt \geq \frac{\varepsilon} 2, \qquad  \text{for} \qquad 
H =\frac{4(1-\sigma)} {\pi} \log \log \p{\frac{C(1-\sigma)}{\varepsilon}}.
\end{gather*}
 \end{lem}
\begin{proof} This follows from Lemma 3. \end{proof}
   
We will now use our result from \cite{Andersson6}. 
\begin{lem}
  Assume that $0 <\alpha \leq 1$, and that $|a_n| \leq M$. Then  we have for $0<\delta \leq 0.05$ that
  \begin{gather*}
    \inf_{\sigma>1,T} \int_T^{T+\delta} \abs{\alpha^{-\sigma-it}+\sum_{n=1}^\infty a_n (n+\alpha)^{-\sigma-it}} dt \geq \alpha^{-1} \p{1+\frac {M^2 \alpha} \delta}^{-\frac{7}{6\delta}} 10^{-\frac 9 \delta}.  
  \end{gather*}
 \end{lem}
\begin{proof}  Lemma 15 in \cite{Andersson6} is in fact stated for $M=1$, but the same proof holds for any $M>0$. \end{proof}

We will now state a Lemma that by the fact that the Hurwitz-zeta function can be approximated by a Dirichlet polynomial yields Theorem 3. 
\begin{lem}
 Let $A(s,\alpha)$ be a Dirichlet polynomial such that 
    \begin{gather*}
     A(s,\alpha)= \alpha^{-s} +\sum_{n=1}^N a_n (n+\alpha)^{-s}
  \end{gather*}
   and $\abs{a_n} \leq M$. 
Then we have for $0<\sigma<1$ that
\begin{multline*}
     \inf_T  
     \int_{-\infty}^{\infty} \exp \p{-2\cosh \frac{\pi t} {2(1-\sigma)}} \int_{t}^{t+\delta} \abs{A(\sigma+i\tau+iT,\alpha)} d\tau dt \geq \\ \geq \frac \pi {9 \alpha (1-\sigma)}  \p{1+ \frac {194 \alpha M^2} {\delta} }^{-\frac{7}{6\delta}} 10^{-\frac 9 \delta}.  
  \end{multline*}
\end{lem} 
\begin{proof}
  By convoluting $A(s,\alpha)$ with Ramachandra's kernal,
 we get in the same way as Theorem 5\footnote{This corresponds to $\alpha=1$.}.
\begin{gather*} 
  \frac {1} {2 \lambda} \int_{-\infty}^{\infty } \alpha^{s+it} 
    A(s+it,\alpha) \exp \p{-2 \cosh \p{\frac t \lambda}} dt = \\ \begin{split}&= \sum_{n=1}^N a_n K_{i \lambda (\log (n+\alpha)-\log \alpha)}(2) (n+\alpha)^{-s}, 
\\ &=  \sum_{n=1}^N b_n  (n+\alpha)^{-s}. \end{split}
\end{gather*}
The result follows by choosing 
$$ \lambda = \frac 2 \pi (1-\sigma)$$
and noticing that $K_0(2) \geq 1/9$ and the fact that
$$
 \sup_{t \geq 0}  \frac{\abs{K_{it}(2)e^{\pi t/2}}}{K_0(2)}=13.917, \qquad 13.917^2<194.
$$
That this expression is bounded follows from Lemma 2 and the constant follows from finding the maximum in Maple\footnote{This should possibly be done in a more rigid manner (without Maple).}.
\end{proof}

\begin{lem}
 Let $A(s,\alpha)$ be a Dirichlet polynomial such that 
    \begin{gather*}
     A(s,\alpha)= \alpha^{-s} +\sum_{n=1}^N a_n (n+\alpha)^{-s}, \qquad \qquad (N \geq 16),
  \end{gather*}
   and $\abs{a_n} \leq 1$. 
Then we have for $1/2 \leq \sigma<1$ that
\begin{gather*}
    \inf_T  \int_0^{\delta+\Delta} \abs{A(\sigma+it,\alpha)}dt \geq  \frac 1 {4 \alpha \delta  (1-\sigma)}  \p{1+ \frac {194 \alpha} {\delta} }^{-\frac{7}{6\delta}} 10^{-\frac 9 \delta}, \\ \intertext{where}
 \Delta = \frac{4}{\pi} (1-\sigma) \log \log N, \, \, \, \, \text{whenever}\, \, \, \,
     N \geq (1-\sigma) \delta^2 \p{1+ \frac {194\alpha}{\delta}}^{7/(3 \delta)} 10^{18/\delta}.
  \end{gather*}
\end{lem} 
\begin{proof}
  Let
  \begin{gather} \label{uiq}
     B(t)= \int_{t}^{t+\delta} \abs{A(\sigma+i\tau+it,\alpha)} d\tau.
    \\  \intertext{By estimating the Dirichlet polynomial $A(\sigma+i\tau+it)$ by its absolute values and integrating over $\tau$ it follows that}
  \notag     B(t) \leq \delta  \frac{N^{1-\sigma}}{1-\sigma}.
\intertext{By Lemma 4 and Lemma 6 it follows that} \label{utrt}
 \int_{0}^H B (t) dt \geq    \frac \pi {18 \alpha (1-\sigma)}  \p{1+ \frac {194 \alpha} {\delta} }^{-\frac{7}{6\delta}} 10^{-\frac 9 \delta},
  \intertext{for} \notag
H =\frac{4(1-\sigma)} {\pi} \log \log \p{\delta \frac{N^{1-\sigma}}{1-\sigma} \frac{(1-\sigma)}{\frac \pi {9 \alpha (1-\sigma)}  \p{1+ \frac {194 \alpha} {\delta} }^{-\frac{7}{6\delta}} 10^{-\frac 9 \delta}}}.
\end{gather}
The fact that 
\begin{gather*}
  H \leq \Delta=\frac {4(1-\sigma)}\pi \log \log N
\end{gather*}
follows from the fact that for $1/2 \leq \sigma \leq 1$ and $N \geq 4$ we have that $N^{1-\sigma} \leq  \sqrt N$, which follows from the lower bound for $N$ in the Lemma.
By  \eqref{uiq} and the triangle inequality we obtain that
$$
 \int_{0}^H B(t)  dt  \leq \delta \int_0^{H+\delta}  \abs{A(it+iT)} dt\leq \delta \int_0^{\Delta+\delta}  \abs{A(it+iT)} dt..
$$
The lemma follows by combining this with \eqref{utrt}.
\end{proof}

\begin{lem}
 Let $A(s,\alpha)$ be a Dirichlet polynomial such that 
    \begin{gather*}
     A(s,\alpha)= \alpha^{-s} +\sum_{n=1}^N a_n (n+\alpha)^{-s}
  \end{gather*}
   and $\abs{a_n} \leq 1$. 
Then we have for $0<\sigma<1$ that. Furthermore let $\delta>0$ and choose $0<\varepsilon<1$ so that $0<\varepsilon  H <0.05$. Then
\begin{gather*}
    \inf_T  \int_T^{T+H} \abs{A(\sigma+it,\alpha)}dt \geq  \frac 1 {4 \alpha (1-\sigma)}  \p{1+ \frac {194 \alpha} {H \varepsilon} }^{-\frac{7}{6H \varepsilon}} 10^{-\frac 9 {H \varepsilon}}, \\ \intertext{for}
 \sigma \geq 1-\frac{\pi H(1-\varepsilon)}{4 (\log \log N+1)}.
  \end{gather*}
 \end{lem} 
\begin{proof}
 This follows by choosing $\delta=\varepsilon H,$
in Lemma 7, since the inequality for $\sigma$ gives us that
$\Delta\leq (1-\varepsilon)H,$
and thus $\delta+\Delta \leq H.$
\end{proof}

\subsection{Proof of Theorem 3}
Theorem 3 follows from Lemma 8 since $\zeta(\sigma+iT,\alpha)$ can be approximated by a Dirichlet polynomial of length $T$, similarly to \eqref{approx}, so we can choose $N=T$ and  
$\alpha=1$ minimizes the right hand side of the inequality in Lemma 8. 
\qed

\section{Further research and open problems}
From \eqref{ity} it follows that 
$$
  \int_{T}^{T+H} \abs{\zeta(1+it)} dt \gg \max(H^2,H).
$$
We may ask how far it is possible to extend this result to the critical strip.
\begin{prob}
 Is it possible to remove the $H^\varepsilon$ on the right hand side of Theorem 1 unconditionally?
\end{prob}
Theorem 2 shows that this can be done assuming the Riemann hypothesis and unconditionally for a shorter range in $\sigma$. It does not seem as  the methods of this paper can do this unconditionally for the full range of $\sigma$ in the theorem.

The general problem to find a lower bound for
$$\int_T^{T+H} \abs{\zeta(\sigma+it)} dt$$
is quite important for $1/2<\sigma<1$, and has applications on e.g. the multiplicity of zeta-zeros, see Ivi\'c \cite{Ivic3}. Our results give good estimates when $1-\sigma \leq H (\log \log T)^{-1-\varepsilon}$. In particular  we see that the range of $\sigma$ where we have good estimates depends on both $T$ and $H$. It is therefore natural to ask: 
\begin{prob}
  Is it possible to remove the dependence on $H$ in Theorem 1 for the range of $\sigma$ where the inequality is valid? Can we prove that
 $$\int_T^{T+H} \abs{\zeta(\sigma+it)} dt \gg \min(H^{2+\varepsilon},H), \qquad \p{\sigma \geq 1-(\log \log T)^{-1-\varepsilon}}?$$
\end{prob}
It is clear that this can be done if we just consider sufficiently large $T$.
\begin{cor}
  Let $\varepsilon>0$. Then we have that
  \begin{gather*}
   \liminf_{T \to \infty} \int_T^{T+H} \abs{\zeta(\sigma+it)} dt \gg \min(H^{2+\varepsilon},H), \, \, \, \,  \p{\sigma \geq 1-(\log \log T)^{-1-\varepsilon}}.
  \end{gather*} 
\end{cor}
\begin{proof}
 This follows from Theorem 1, by choosing $\varepsilon$ in Theorem 1 to be half of the $\varepsilon$  in Corollary 1 and using the fact that $\lim_{T \to \infty} (\log \log T)^{-\varepsilon/2}=0$.
\end{proof}
Corresponding corollaries also follows from Theorem 3-4 by the same proof method (we remark that we already stated Theorem 2 in this manner). It is not too difficult to give some explicit estimate of $T$ depending on $H$ where problem 2 can be solved, e.g. we may prove the lower bound in Problem 2 for each  
$$
 T \geq \exp \p {\exp \p{H^{-1/\varepsilon}}}.
$$
However, we do not seem to get as sharp bounds for smaller $T$. It is easy to see that if we can answer problem 2 in the affirmative it would follow that all the zeroes $\rho=\sigma+it$ of the Riemann zeta-function for $\sigma \geq 1-(\log \log t)^{-1-\varepsilon}$ are simple.

While it might be possible to remove the dependence between $\sigma$ and $H$ in Theorem 1 as suggested by Problem 2,  universality results on vertical lines, see e.g. \cite{Andersson6} implies that it is not possible to prove Theorem 1 for $\Re(s)>\sigma$ for any fixed $\sigma<1$ and $H>0$, so it is not possible to remove the dependence between $\sigma$ and $T$.

\bibliographystyle{plain}

\begin{thebibliography}{10}

\bibitem{Andersson}
J.~Andersson.
\newblock Disproof of some conjectures of {K}. {R}amachandra.
\newblock {\em Hardy-Ramanujan J.}, 22:2--7, 1999.

\bibitem{Andersson3}
J.~Andersson.
\newblock On the zeta function on the line $\Re(s)=1$.
\newblock \href{https://arxiv.org/abs/1207.4336}{\texttt{arXiv:1207.4336 [math.NT]}} 

\bibitem{Andersson6}
J.~Andersson.
\newblock Nonuniversality on the critical line.
\newblock \href{https://arxiv.org/abs/1207.4927}{\texttt{arXiv:1207.4927 [math.NT].}}

\bibitem{Andersson4}
J.~Andersson.
\newblock On a problem of {R}amachandra and approximation of functions by
  {D}irichlet polynomials with bounded coefficients.
\newblock \href{https://arxiv.org/abs/1207.4624}{\texttt{arXiv:1207.4624 [math.NT]
.}} 

\bibitem{Andersson15}
J.~Andersson.
\newblock On generalized {H}ardy classes of {D}irichlet series.
\newblock \href{https://arxiv.org/abs/1207.5337}{\texttt{arXiv:1207.5337 [math.CV].}}

\bibitem{Basset}
A.~B. Basset.
\newblock On the potentials of the surfaces formed by the revolution of
  {Lima}cons and {C}ardioids about their axes.
\newblock {\em Proc. Camb. Phil. Soc}, vi:2--19, 1889.

\bibitem{Ford}
K.~Ford.
\newblock Vinogradov's integral and bounds for the {R}iemann zeta function.
\newblock {\em Proc. London Math. Soc. (3)}, 85(3):565--633, 2002.

\bibitem{Ivic2}
A.~Ivi{\'c}.
\newblock {\em Lectures on mean values of the {R}iemann zeta function},
  volume~82 of {\em Tata Institute of Fundamental Research Lectures on
  Mathematics and Physics}.
\newblock Published for the Tata Institute of Fundamental Research, Bombay,
  1991.

\bibitem{Ivic3}
A.~Ivi{\'c}.
\newblock On the multiplicity of zeros of the zeta-function.
\newblock {\em Bull. Cl. Sci. Math. Nat. Sci. Math.}, (24):119--132, 1999.

\bibitem{Ivic}
A.~Ivi{\'c}.
\newblock {\em The {R}iemann zeta-function}.
\newblock Dover Publications Inc., Mineola, NY, 2003.
\newblock Theory and applications, Reprint of the 1985 original [Wiley, New
  York; MR0792089 (87d:11062)].

\bibitem{Koosis}
P.~Koosis.
\newblock {\em The logarithmic integral. {I}}, volume~12 of {\em Cambridge
  Studies in Advanced Mathematics}.
\newblock Cambridge University Press, Cambridge, 1998.
\newblock Corrected reprint of the 1988 original.
\bibitem{Korobov}
N.~M. Korobov.
\newblock Estimates of trigonometric sums and their applications.
\newblock {\em Uspehi Mat. Nauk}, 13(4 (82)):185--192, 1958.

\bibitem{Macdonald}
H.~M. Macdonald.
\newblock Zeroes of the {B}essel {F}unctions.
\newblock {\em Proc. London. Math. Soc}, xxx:165--179, 1899.

\bibitem{PalWien}
R.E.A.~C. Paley and N.~Wiener.
\newblock {\em Fourier transforms in the complex domain}, volume~19 of {\em
  American Mathematical Society Colloquium Publications}.
\newblock American Mathematical Society, Providence, RI, 1987.
\newblock Reprint of the 1934 original.

\bibitem{Ramachandra6}
K.~Ramachandra.
\newblock Mean-value of the {R}iemann zeta-function and other remarks. {I}.
\newblock In {\em Topics in classical number theory, {V}ol. {I}, {II}
  ({B}udapest, 1981)}, volume~34 of {\em Colloq. Math. Soc. J\'anos Bolyai},
  pages 1317--1347. North-Holland, Amsterdam, 1984.

\bibitem{Ramachandra2}
K.~Ramachandra.
\newblock On {R}iemann zeta-function and allied questions.
\newblock {\em Ast\'erisque}, (209):57--72, 1992.
\newblock Journ{\'e}es Arithm{\'e}tiques, 1991 (Geneva).

\bibitem{Ramachandra}
K.~Ramachandra.
\newblock {\em On the mean-value and omega-theorems for the {R}iemann
  zeta-function}, volume~85 of {\em Tata Institute of Fundamental Research
  Lectures on Mathematics and Physics}.
\newblock Published for the Tata Institute of Fundamental Research, Bombay,
  1995.

\bibitem{Schafli}
L.~Sch\"afli.
\newblock Sopra un teorema di {J}acobi recato a forma pi{\'u} generale ed
  applicata alla funzione cilindrica.
\newblock {\em Ann.di.Mat}, v(2):199--205, 1873.

\bibitem{Selberg}
A.~Selberg.
\newblock On the zeros of {R}iemann's zeta-function.
\newblock {\em Skr. Norske Vid. Akad. Oslo I.}, 1942(10):59, 1942.
\bibitem{Titchmarsh}
E.~C. Titchmarsh.
\newblock {\em The theory of the {R}iemann zeta-function}.
\newblock The Clarendon Press Oxford University Press, New York, second
  edition, 1986.
\newblock Edited and with a preface by D. R. Heath-Brown.

\bibitem{Vinogradov}
I.~M. Vinogradov.
\newblock A new estimate of the function {$\zeta (1+it)$}.
\newblock {\em Izv. Akad. Nauk SSSR. Ser. Mat.}, 22:161--164, 1958.

\bibitem{Watson}
G.~N. Watson.
\newblock {\em A treatise on the theory of {B}essel functions}.
\newblock Cambridge Mathematical Library. Cambridge University Press,
  Cambridge, 1995.
\newblock Reprint of the second (1944) edition.

\end{thebibliography}

\end{document}